\documentclass[a4paper, notitlepage, 12pt]{article}
\usepackage{amssymb}
\usepackage{amsfonts}
\usepackage{amsmath}
\usepackage{indentfirst}
\usepackage{theorem}

\headheight 0cm \headsep 0cm {\theorembodyfont{\upshape}
\newtheorem{remark}{Remark}[section]

}
\newtheorem{theorem}{Theorem}[section]
\newtheorem{corollary}{Corollary}[section]
\newtheorem{definition}{Definition}[section]
\newtheorem{lemma}{Lemma}[section]

\newenvironment{proof}[1][Proof]{\noindent\textbf{#1.} }{\ \rule{0.5em}{0.5em}}

\makeatletter \@addtoreset{equation}{section}
\renewcommand{\theequation}{\thesection.\@arabic\c@equation}
\@addtoreset{figure}{section}
\renewcommand{\thefigure}{\thesection.\@arabic\c@figure}
\@addtoreset{table}{section}
\renewcommand{\thetable}{\thesection.\@arabic\c@table}
\makeatother
\input{tcilatex}

\begin{document}

\title{An Intrinsic Impulse Observability Criterion for Descriptor System
\thanks{%
Supported by Harbin Institute of Technology Science Foundation under grant
HITC200712.}}
\author{Zhibin Yan \\
%EndAName
{\small Center for control and Guidance,\ Harbin Institute of Technology,\
Harbin, 150001, China}\\
{\small e-mail: zbyan@hit.edu.cn}}
\date{}
\maketitle

\begin{abstract}
Analyzing the order of unobservable impulse in descriptor system leads to a
new testing criterion for impulse observability, both the statement and the
proof of which use only the original system data.


\end{abstract}

\section{Introduction\label{SecOne}}

We consider the impulse observability of descriptor linear system \cite%
{Lewis1986}, \cite{Dai1989}
\begin{equation}
\begin{array}{rll}
E\dot{x}(t) & = & Ax(t) \\
y(t) & = & Cx(t)%
\end{array}
\label{Sys3}
\end{equation}%
where $E,$ $A\in {\mathbb{R}}^{n\times n},$ $C\in {\mathbb{R}}^{m\times n}.$
Matrix $E$ is singular, but matrix pencil $sE-A$ is regular, i.e., $\det
(sE-A)$ is a nonzero polynomial on $s\in \mathbb{C}$ \cite{Gantmacher}.
Comparing with standard linear system ($E=I,$ the identity matrix),
descriptor one is featured by having impulse behavior. The underlying
mathematics for this phenomenon is that the following initial value problem
of differential-algebraic equation
\begin{equation}
E\dot{x}(t)=Ax(t),\,t\geq 0;\quad x(0)=w  \label{initialvalueproblem}
\end{equation}%
has no solution generally in the sense of classical differentiable function,
and a generalized solution, as a mathematical model of the state response of
the system to initial value, in the sense of distribution has to be adopted.
It interprets the impulse behavior that the distributional solution may
contain a linear combination of Dirac delta distribution $\delta (t)$\ and
its distributional derivatives $\delta ^{(k)}(t),$ $k=1,2,\ldots $. This
linear combination is called impulsive term in the state response for
convenience in following. For details, see \cite[pp. 16--22]{Dai1989} and
the references therein. More recent works about distributional solution are
\cite{Muller2005}--\cite{Yan2008}.

System (\ref{Sys3}) is called impulse observable (see, e.g., \cite[p. 43]%
{Dai1989}, \cite{Ozccaldiran1992}--\cite{Ishihara2001}, \cite{Yan2006}), if
for arbitrarily given but unknown initial value $x(0),$ the impulsive term
in state response $x(t),$ $t\geq 0$ can be uniquely determined out from the
measured, therefore known, output response $y(t),$ $t\geq 0.$

A subtle point for impulse analysis in descriptor system is that explicit
expressions of the impulsive terms in state and output responses directly
using the system data, i.e., $\{E,A,C\}$ for system (\ref{Sys3}) discussed
here, do not exist. As a result, the existing impulse observability
criteria, mainly collected in \cite[Theorem 2-3.4, p. 43 ]{Dai1989} and the
dual form of \cite[Theorem 4.2, p. 20]{Lewis1986}, and/or their proofs have
to employ indirect data which comes from some transformations to the
original system, e.g., slow-fast decomposition based on Weierstrass
canonical form, descriptor system structure algorithm \cite{Lewis1986}, etc.
Besides causing numerical difficulties, to deal with which some condensed
forms based on numerically reliable orthogonal matrix transformations are
developed \cite{Mehrmann1999}--\cite{Mehrmann2008}, such type of answer to a
control theory problem can not satisfy a theoretical interest as well.

In this note, we are interested in obtaining an \textbf{intrinsic} impulse
observability criterion in the sense that both its statement and its proof
use the original system data directly. This concern is also motivated by
seeking for a criterion with clear dynamical interpretation. The obtained
new criterion is Theorem \ref{ThNewCri}, the dynamical interpretation of
which is the nonexistence of unobservable impulse of specific order, see
Theorem \ref{ThRjieCri}. Main existing criteria for impulse observability in
literature are special cases of our new one, see Corollaries \ref{FirstNew}
and \ref{SecondNew}.

\section{Main Results}

The frequency domain form of (\ref{Sys3}) with initial state $x(0)=w$ can be
written as
\begin{equation}
\left[
\begin{array}{c}
sE-A \\
C%
\end{array}%
\right] X(s)=\left[
\begin{array}{c}
Ew \\
Y(s)%
\end{array}%
\right] .  \label{Sys4}
\end{equation}%
where $X,Y$ denote the Laplace transforms of $x,y$ respectively.

We use notations ${\mathbb{R}}^{n}(s)$ and ${\mathbb{R}}^{n}[s]$ to denote
the set of $n$-dimensional rational fraction vectors and the set of $n$%
-dimensional polynomial vectors respectively. Note that any rational
fraction can be uniquely expressed as the sum of a strictly proper fraction
and a polynomial. \textbf{In following for }$F(s)\in {\mathbb{R}}^{n}(s)$%
\textbf{\ we always use }$F_{\mathrm{A}}(s)$\textbf{\ and }$F_{\mathrm{P}%
}(s) $\textbf{\ to denote the strictly proper fraction part and polynomial
part respectively for notation convenience.}

For clarity we use the following Definition \cite{Yan2007}.

\begin{definition}
\label{ImpulseOrder}Let $v_{i}\in {\mathbb{R}}^{n},$ $i=0,1,\ldots ,p$ with $%
v_{p}\neq 0.$ Then the distribution%
\begin{equation*}
\lambda (t)=\delta (t)v_{0}+\delta ^{(1)}(t)v_{1}+\cdots +\delta
^{(p)}(t)v_{p}
\end{equation*}%
is called an $n$-dimensional impulse of order $p,$ denoted by $\deg (\lambda
)=p.$ The polynomial vector $\tsum\nolimits_{i=0}^{p}s^{i}v_{i}$ $\in $ ${%
\mathbb{R}}^{n}[s]$ in complex variable $s$ of degree $p,$ which is Laplace
transform of impulse, is called frequency domain form of impulse, or simply,
impulse.
\end{definition}

The Laplace inverse transform of a strictly proper fraction is a usual
smooth function (a combination of exponential, triangular and polynomial
functions of time variable $t.$ see \cite{Oppenheim2002}). Therefore impulse
observability of the system (\ref{Sys3}) means, in frequency domain
language, that the polynomial part $X_{\mathrm{P}}(s)$ in $X(s)$ can be
uniquely determined out from $Y(s)$. We write this fact as definition of
impulse observability for clarity, although \textquotedblleft
common\textquotedblright\ definition is not so stated \cite[p. 43]{Dai1989},
\cite{Dai1989b}, \cite{Ozccaldiran1992}, \cite{Hou1999}, \cite{Ishihara2001}%
, \cite{Yan2006}.

\begin{definition}
\label{Lemmacri}The system (\ref{Sys3}) is impulse observable, if the
following equation%
\begin{equation}
\left[
\begin{array}{c}
sE-A \\
C%
\end{array}%
\right] X(s)=\left[
\begin{array}{c}
Ew \\
0%
\end{array}%
\right]  \label{eqfrac}
\end{equation}%
on $(w,$ $X(s))\in {\mathbb{R}}^{n}\times {\mathbb{R}}^{n}(s)$ has no
solution with $X_{\mathrm{P}}(s)$ nonzero.
\end{definition}

\begin{remark}
In other words, from $Y(s)=CX(s)=0$ one can conclude that $X(s)$ does not
contain polynomial part (frequency domain form of impulse, see Definition %
\ref{ImpulseOrder}).
\end{remark}

\begin{remark}
We adopt the equation viewpoint. Here a deliberate point is that both $w$
and $X(s)$ are seen undetermined simultaneously. It is this insight that
results in an equal treating with $p_{-1}$ and$\ p_{0},\cdots ,p_{r}$ in (%
\ref{EqOhCN})\ and the introducing of the matrix (\ref{ImObserMatrixK}).
\end{remark}

\begin{theorem}
\label{Theoerempoly}The system (\ref{Sys3}) is not impulse observable, if
and only if the following equation%
\begin{equation}
\left[
\begin{array}{c}
sE-A \\
C%
\end{array}%
\right] P(s)=\left[
\begin{array}{c}
Ev \\
0%
\end{array}%
\right]  \label{equnobserv}
\end{equation}%
on $(v,$ $P(s))\in {\mathbb{R}}^{n}\times {\mathbb{R}}^{n}[s]$ has a
solution with $P(s)$ nonzero.
\end{theorem}

Note the difference of ${\mathbb{R}}^{n}(s)$ and ${\mathbb{R}}^{n}[s]$.

\begin{proof}
The sufficiency is obvious by Definition \ref{Lemmacri} since $P(s)$ $\in $ $%
{\mathbb{R}}^{n}[s]$ $\subset $ ${\mathbb{R}}^{n}(s).$

Necessity. By Definition \ref{Lemmacri}, there exists $(w,$ $X(s))\in {%
\mathbb{R}}^{n}\times {\mathbb{R}}^{n}(s)$ with $X_{\mathrm{P}}(s)$ nonzero
such that%
\begin{equation*}
\left[
\begin{array}{c}
sE-A \\
C%
\end{array}%
\right] (X_{\mathrm{A}}(s)+X_{\mathrm{P}}(s)=\left[
\begin{array}{c}
Ew \\
0%
\end{array}%
\right] .
\end{equation*}%
Firstly, $0=CX_{\mathrm{A}}(s)+CX_{\mathrm{P}}(s)\in {\mathbb{R}}^{m}(s)$
implies
\begin{equation}
CX_{\mathrm{P}}(s)=0.  \label{ProofEq1}
\end{equation}%
Secondly, the limit $\lim_{s\rightarrow \infty }sX_{\mathrm{A}}(s)$ exists,
denoted by $q,$ and moreover $\lim_{s\rightarrow \infty }(sE-A)X_{\mathrm{A}%
}(s)$ $=$ $Eq.$ Therefore
\begin{eqnarray*}
Ew &=&(sE-A)(X_{\mathrm{A}}(s)+X_{\mathrm{P}}(s) \\
&=&[(sE-A)X_{\mathrm{A}}(s)-Eq]+[Eq+(sE-A)X_{\mathrm{P}}(s)]
\end{eqnarray*}%
forms a decomposition of strictly proper fraction plus polynomial. The
uniqueness of such decomposition implies $Ew$ $=$ $[Eq+(sE-A)X_{\mathrm{P}%
}(s)],$ which is equivalent to%
\begin{equation}
(sE-A)X_{\mathrm{P}}(s)=E(w-q).  \label{ProofEq2}
\end{equation}%
It follows from (\ref{ProofEq1}) and (\ref{ProofEq2}) that $(w-q,$ $X_{%
\mathrm{P}}(s))\in {\mathbb{R}}^{n}\times {\mathbb{R}}^{n}[s]$ is a solution
of (\ref{equnobserv}).
\end{proof}

We introduce the following technical notion.

\begin{definition}
\label{Defrorder}Let $(v,$ $P(s))$ $\in $ ${\mathbb{R}}^{n}$ $\times $ ${%
\mathbb{R}}^{n}[s]$ with $P(s)$ nonzero and $\deg (P(s))=r$ be a solution of
the equation (\ref{equnobserv}). Then $v$ is called an unobservable
impulsive initial state of the system (\ref{Sys3}), and $P(s)$ is called an
unobservable impulse of order $r$ of the system (\ref{Sys3}).
\end{definition}

\begin{lemma}
\label{propnorder}The system (\ref{Sys3}) is impulse observable, if and only
if it has no unobservable impulse of order $\leq n-1.$
\end{lemma}

\begin{proof}
It follows from the regularity of the pencil $sE-A$ that the solution of (%
\ref{equnobserv}), if exists, will be of $\deg (P(s))\leq n-1.$
\end{proof}

\begin{theorem}
\label{Lemmatwo}Let $(v,$ $P(s))\in {\mathbb{R}}^{n}\times {\mathbb{R}}%
^{n}[s]$ be a solution of (\ref{equnobserv}) with $\deg (P(s))=r\geq 1$ and
write%
\begin{equation}
P(s)=p_{0}-sp_{1}+\cdots +(-s)^{r}p_{r}.  \label{Expanssion}
\end{equation}%
Then $(p_{r-1},$ $p_{r})\in {\mathbb{R}}^{n}\times {\mathbb{R}}^{n}[s]$ is a
solution of (\ref{equnobserv}) as well, where $p_{r}\in {\mathbb{R}}^{n}[s]$
is of $\deg (p_{r})=0$ as a polynomial vector.
\end{theorem}

\begin{proof}
Combining (\ref{equnobserv}) and (\ref{Expanssion}), we have%
\begin{equation*}
\sum_{i=0}^{r}\left[
\begin{array}{c}
s^{i+1}E-s^{i}A \\
s^{i}C%
\end{array}%
\right] (-1)^{i}p_{i}=\left[
\begin{array}{c}
Ev \\
0%
\end{array}%
\right] .
\end{equation*}%
Differentiating $r$ times gives%
\begin{equation*}
\left[
\begin{array}{c}
r!E-0 \\
0%
\end{array}%
\right] (-1)^{r-1}p_{r-1}+\left[
\begin{array}{c}
(r+1)!sE-r!A \\
r!C%
\end{array}%
\right] (-1)^{r}p_{r}=\left[
\begin{array}{c}
0 \\
0%
\end{array}%
\right]
\end{equation*}%
and further%
\begin{equation}
\left[
\begin{array}{c}
sE-A \\
C%
\end{array}%
\right] p_{r}=\left[
\begin{array}{c}
Ep_{r-1} \\
0%
\end{array}%
\right] .  \label{OrderReduc}
\end{equation}%
Note that $(r+1)sEp_{r}=0$ implies $Ep_{r}=0.$
\end{proof}

From an unobservable impulse $P(s)$ of order $r$ to initial value $v,$
Theorem \ref{Lemmatwo} constructs an unobservable one $p_{r}$ of order zero
to initial value $p_{r-1}$.

Form the following partitioned matrix%
\begin{equation}
{\mathcal{O}}_{k}(E,A,C)=\left[
\begin{array}{cccc}
E & A &  &  \\
& E & \ddots &  \\
&  & \ddots & A \\
&  &  & E \\
0 & C &  &  \\
& \ddots & \ddots &  \\
&  & 0 & C%
\end{array}%
\right]  \label{ImObserMatrixK}
\end{equation}%
for $k=2,\ldots ,n+1,$ where the other blocks not appearing are zero. It has
$k$ block columns and $2k-1$ block rows, and then has the size $%
(kn+(k-1)m)\times kn.$

\begin{theorem}
\label{ThRjieCri}The system (\ref{Sys3}) has no unobservable impulse of
order $\leq r$ if and only if
\begin{equation}
\limfunc{rank}({\mathcal{O}}_{r+2}(E,A,C))=n(r+1)+\limfunc{rank}(E).
\label{RjieRank}
\end{equation}
\end{theorem}

\begin{proof}
Let $\deg (P(s))\leq r$ and write $P(s)$ $=$ $p_{0}-sp_{1}+\cdots
+(-s)^{r}p_{r}.$ Then (\ref{equnobserv}) is equivalent to the following $%
(r+2)+(r+1)=2r+3$ equations%
\begin{equation}
Ep_{i}+Ap_{i+1}=0,\text{ }i=-1,0,\ldots ,r,  \label{EqNpi}
\end{equation}%
\begin{equation}
Cp_{i}=0,\text{ }i=0,\ldots ,r  \label{EqCpi}
\end{equation}%
on $(p_{-1},p_{0},\cdots ,p_{r})\in {\mathbb{R}}^{n}\times \cdots \times {%
\mathbb{R}}^{n},$ where we denote $p_{-1}=v$ and $p_{r+1}=0$ for notation
convenience. The group of equations (\ref{EqNpi}) and (\ref{EqCpi}) can be
rewritten into the following matrix form%
\begin{equation}
{\mathcal{O}}_{r+2}(E,A,C)\left[
\begin{array}{c}
p_{-1} \\
p_{0} \\
\vdots \\
p_{r}%
\end{array}%
\right] =0.  \label{EqOhCN}
\end{equation}%
Therefore the system (\ref{Sys3}) has no $r$ order unobservable impulse if
and only if all solutions of (\ref{EqOhCN}) satisfy $p_{i}=0,$ $i=0,1,\ldots
,r,$ i.e.,
\begin{equation}
\left\{ \left[
\begin{array}{c}
p_{-1} \\
p_{0} \\
\vdots \\
p_{r}%
\end{array}%
\right] :{\mathcal{O}}_{r+2}(E,A,C)\left[
\begin{array}{c}
p_{-1} \\
p_{0} \\
\vdots \\
p_{r}%
\end{array}%
\right] =0\right\} =\left\{ \left[
\begin{array}{c}
p_{-1} \\
0 \\
\vdots \\
0%
\end{array}%
\right] :{\mathcal{O}}_{r+2}(E,A,C)\left[
\begin{array}{c}
p_{-1} \\
0 \\
\vdots \\
0%
\end{array}%
\right] =0\right\} .  \label{SETRELA}
\end{equation}%
It is easy to see that
\begin{equation*}
{\mathcal{O}}_{r+2}(E,A,C)\left[
\begin{array}{c}
p_{-1} \\
0 \\
\vdots \\
0%
\end{array}%
\right] =Ep_{-1}.
\end{equation*}%
Then by computing dimensions of solution spaces of linear equations, (\ref%
{SETRELA}) gives%
\begin{equation*}
n(r+2)-\func{rank}({\mathcal{O}}_{r+2}(E,A,C))=n-\func{rank}(E)
\end{equation*}%
and the result follows immediately.
\end{proof}

\begin{theorem}
\label{ThNewCri}The following statements are equivalent:

1). The system (\ref{Sys3}) is impulse observable;

2). $\limfunc{rank}({\mathcal{O}}_{r+2}(E,A,C))=n(r+1)+\limfunc{rank}(E)$
for some one $r\in \{0,1,\cdots ,n-1\};$

3). $\limfunc{rank}({\mathcal{O}}_{r+2}(E,A,C))=n(r+1)+\limfunc{rank}(E)$
for each one $r\in \{0,1,\cdots ,n-1\}.$
\end{theorem}

\begin{proof}
Follows from Lemma \ref{propnorder}, Theorems \ref{Lemmatwo} and \ref%
{ThRjieCri}.
\end{proof}

\begin{remark}
Up till now, all statements and proofs use only the original system data,
not involved in any transformation to the system.
\end{remark}

\section{Implying Existing Results}

First, taking $r=0$, the condition (\ref{RjieRank}) gives the following
consequence.

\begin{corollary}
\label{FirstNew}The system (\ref{Sys3}) is impulse observable, if and only
if $\limfunc{rank}({\mathcal{O}}_{2}(E,A,C))=n+\limfunc{rank}(E),$ i.e.,
\begin{equation}
\limfunc{rank}\left[
\begin{array}{cc}
E & A \\
0 & E \\
0 & C%
\end{array}%
\right] =n+\limfunc{rank}(E).  \label{Dai}
\end{equation}
\end{corollary}

\begin{remark}
The condition (\ref{Dai}) is a well known criterion for impulse
observability \cite[Eq. (2-3.6), p. 44]{Dai1989}, which is featured by using
the original system data. To our knowledge, it is the only criterion of this
feature in literature. Existing proofs (see \cite[p. 44]{Dai1989}, \cite%
{Ishihara2001}, etc.) rely on some decompositions. Now we see that it
guarantees the nonexistence of unobservable impulse of zero order. In
quantitative aspect, equation%
\begin{equation*}
\left[
\begin{array}{cc}
E & A \\
0 & E \\
0 & C%
\end{array}%
\right] \left[
\begin{array}{c}
p_{-1} \\
p_{0}%
\end{array}%
\right] =0
\end{equation*}%
provides total information about unobservable impulsive initial states of
zero order.
\end{remark}

Now we consider the Weierstrass canonical decomposition%
\begin{equation}
\left[
\begin{array}{cc}
T & 0 \\
0 & I%
\end{array}%
\right] \left[
\begin{array}{c}
sE-A \\
C%
\end{array}%
\right] S=\left[
\begin{array}{cc}
sI_{n_{1}}-A_{1} & 0 \\
0 & sN-I_{n_{2}} \\
C_{1} & C_{2}%
\end{array}%
\right] ,  \label{Weiers}
\end{equation}%
where $T,$ $S\in {\mathbb{R}}^{n\times n}$ are invertible, $n_{1}+n_{2}=n$
and $N$ is nilpotent with index $h$ (i.e., $N^{h-1}\neq 0$, but $N^{h}=0$).

\begin{lemma}
\label{LemmaA}The condition (\ref{RjieRank}) holds if and only if
\begin{equation*}
\limfunc{rank}({\mathcal{O}}_{r+2}(N,I_{n_{2}},C_{2}))=n_{2}(r+1)+\limfunc{%
rank}(N).
\end{equation*}
\end{lemma}

\begin{proof}
Straightforward.
\end{proof}

\begin{lemma}
\label{LemmaB}%
\begin{equation*}
\limfunc{rank}({\mathcal{O}}_{r+2}(N,I_{n_{2}},C_{2}))=n_{2}(r+1)+\func{rank}%
\left[
\begin{array}{c}
N^{r+2} \\
C_{2}N \\
C_{2}N^{2} \\
\vdots \\
C_{2}N^{r+1}%
\end{array}%
\right] .
\end{equation*}
\end{lemma}

\begin{proof}
Through a series of elementary row and column transformations, matrix ${%
\mathcal{O}}_{r+2}(N,I_{n_{2}},C_{2})$ can be transformed into%
\begin{equation*}
\left[
\begin{array}{cccc}
0 & I_{n_{2}} &  &  \\
& 0 & \ddots &  \\
&  & \ddots & I_{n_{2}} \\
N^{r+2} &  &  & 0 \\
C_{2}N & 0 &  &  \\
\vdots &  & \ddots &  \\
C_{2}N^{r+1} &  &  & 0%
\end{array}%
\right]
\end{equation*}%
and the result follows.
\end{proof}

\begin{lemma}
\label{LemmaC}The condition (\ref{RjieRank}) holds if and only if
\begin{equation}
\func{rank}\left[
\begin{array}{c}
N^{r+2} \\
C_{2}N \\
C_{2}N^{2} \\
\vdots \\
C_{2}N^{r+1}%
\end{array}%
\right] =\func{rank}(N).  \label{FastRankCondi}
\end{equation}
\end{lemma}

\begin{proof}
Follows from Lemmas \ref{LemmaA} and \ref{LemmaB}.
\end{proof}

\begin{corollary}
\label{SecondNew}The following statements are equivalent:

1). The system (\ref{Sys3}) is impulse observable;

2). Condition (\ref{FastRankCondi}) holds for some one $r\in \{0,1,\ldots
,n-1\};$

3). Condition (\ref{FastRankCondi}) holds for each one $r\in \{0,1,\ldots
,n-1\}.$
\end{corollary}

\begin{proof}
Follows from Theorem \ref{ThNewCri} and Lemma \ref{LemmaC} immediately.
\end{proof}

\begin{remark}
When $r\geq h-2$ with $h$ the nilpotency index of $N,$ the rank criterion (%
\ref{FastRankCondi}) gives the condition \cite[Theorem 2-3.4, (iv)]{Dai1989}%
; When $r=0$, it gives another condition \cite[Theorem 2-3.4, (v)]{Dai1989}.
We see that these two well known criteria are only two boundary cases of
Corollary \ref{SecondNew}.
\end{remark}


\begin{thebibliography}{99}
\bibitem{Lewis1986} F. L. Lewis, \textquotedblleft A survey of linear
singular systems,\textquotedblright\ \textit{Circuits, Systems and Signal
Processing}, Vol. 5, No. 1, pp. 3--36, 1986.

\bibitem{Dai1989} L. Dai, \textit{Singular Control Systems}, Lecture Notes
in Control and Information Sciences, Vol. 118, Springer-Verlag Berlin,
Heidelberg, 1989.

\bibitem{Gantmacher} F. R. Gantmacher, \textit{The Theory of Matrices, }Vol.
2. Beijing: Higher Education Press, 1955 (Chinese translation from Russian).

\bibitem{Muller2005} P. C. M\"{u}ller, \textquotedblleft Remark on the
solution of linear time-invariant descriptor systems,\textquotedblright\
\textit{PAMM Proc. Appl. Math. Mech.}, 5, pp. 175--176, 2005.

\bibitem{Yan2005} Z. Yan and G. Duan, \textquotedblleft Time domain solution
to descriptor variable systems,\textquotedblright\ \textit{IEEE Transactions
on Automatic Control}, Vol. 50, No. 11, pp. 1796--1799, 2005.

\bibitem{Yan2006} Z. Yan and G. Duan, \textquotedblleft Impulse analysis of
linear time-varying singular systems,\textquotedblright\ \textit{IEEE
Transactions on Automatic Control}, Vol. 51, No. 12, pp. 1975--1979, 2006.

\bibitem{Yan2007} Z. Yan, \textquotedblleft Geometric analysis of impulse
controllability for descriptor system,\textquotedblright\ \textit{Systems
and Control Letters}, Vol. 56, No. 1, pp. 1--6, 2007.

\bibitem{Yan2008} Z. Yan, \textquotedblleft Consistent-inconsistent
decomposition to initial value problem of descriptor linear
systems,\textquotedblright\ \textit{ZAMM Z. Angew. Math. Mech.} 88(2008),
DOI 200700150, to appear.

\bibitem{Ozccaldiran1992} K. Ozccaldiran, D. W. Fountain, and F. L. Lewis,
\textquotedblleft Some generalized notions of
observability,\textquotedblright\ \textit{IEEE Transactions on Automatic
Control}, Vol. 37, No. 6, pp. 856--860, 1992.

\bibitem{Hou1999} M. Hou and P. C. M\"{u}ller, \textquotedblleft Causal
observability of descriptor systems,\textquotedblright\ \textit{IEEE
Transactions on Automatic Control}, Vol. 44, No. 1, pp. 158--163, 1999.

\bibitem{Hou2004} M. Hou, \textquotedblleft Controllability and elimination
of impulsive modes in descriptor systems,\textquotedblright\ \textit{IEEE
Transactions on Automatic Control}, Vol. 49, No. 10, pp. 1723--1727, 2004.

\bibitem{Wangetal2001} W. Wang and Y. Zou, \textquotedblleft Analysis of
impulsive modes and Luenberger observers for descriptor
systems,\textquotedblright\ \textit{Systems and Control Letters}, Vol. 44,
pp. 347--353, 2001.

\bibitem{Ishihara2001} J. Y. Ishihara and M. H. Terra, \textquotedblleft
Impulse controllability and observability of rectangular descriptor
systems,\textquotedblright\ \textit{IEEE Transactions on Automatic Control},
Vol. 46, No. 6, pp. 991--994, 2001.

\bibitem{Mehrmann1999} A. Bunse-Gerstner, R. Byers, V. Mehrmann, N. K.
Nichols, \textquotedblleft Feedback design for regularizing descriptor
systems,\textquotedblright\ \textit{Linear Algebra Appl.} 299 (1999), no.
1-3, pp. 119--151.

\bibitem{Chu1998} D. L. Chu, H. C. Chan, D. W. C. Ho, \textquotedblleft
Regularization of singular systems by derivative and proportional output
feedback,\textquotedblright\ \textit{SIAM J. Matrix Anal. Appl.,} 19 (1998),
no. 1, pp. 21--38

\bibitem{Chu1999a} D. L. Chu, D. W. C. Ho, \textquotedblleft Necessary and
sufficient conditions for the output feedback regularization of descriptor
systems,\textquotedblright\ \textit{IEEE Trans. Automat. Control,} 44
(1999), no. 2, pp. 405--412.

\bibitem{Chu1999b} D. L. Chu, V. Mehrmann, N. K. Nichols, \textquotedblleft
Minimum norm regularization of descriptor systems by mixed output
feedback,\textquotedblright\ \textit{Linear Algebra Appl.,} 296 (1999), no.
1-3, pp. 39--77.

\bibitem{Byers1997} R. Byers, T. Geerts, and V. Mehrmann, \textquotedblleft
Descriptor systems without controllability at infinity,\textquotedblright\
\textit{SIAM Journal on Control and Optimization}, Vol. 35, No. 2, pp.
462--479, 1997.

\bibitem{Mehrmann2008} P. Losse and V. Mehrmann, \textquotedblleft
Controllability and Observability of Second Order Descriptor
Systems,\textquotedblright\ \textit{SIAM Journal on Control and Optimization}%
, Vol. 47 No. 3, pp. 1351--1379, 2008.

\bibitem{Oppenheim2002} A. V. Oppenheim, A. S. Willsky, and S. H. Nawab,
\textit{Signals and Systems}, Englewood Cliffs, NJ: Prentice-Hall, 1997.
\end{thebibliography}
\end{document}